\documentclass[runningheads, envcountsame, a4paper]{llncs}
\usepackage{amsmath,amssymb}
\usepackage{lmodern}
\usepackage[T1]{fontenc}
\usepackage[utf8]{inputenc} \usepackage{textcomp} 
\usepackage{parskip}
\usepackage{xcolor}
\providecommand{\tightlist}{%
  \setlength{\itemsep}{0pt}\setlength{\parskip}{0pt}}
\usepackage{bm}
\usepackage[capitalize]{cleveref}
\usepackage{thmtools}
\usepackage{tikz}

\usepackage[
backend=biber,
            style=alphabetic,
            url=false,
            maxbibnames=10
]{biblatex}

\renewcommand{\phi}{\varphi}
\renewcommand{\bar}{\overline}

\newcommand{\A}{\mathcal{A}}
\newcommand{\B}{\mathcal{B}}
\newcommand{\mbf}[1]{\mathbf{#1}}
\newcommand{\mf}[1]{\mathfrak{#1}}
\newcommand{\mc}[1]{\mathcal{#1}}
\newcommand{\ra}{\rightarrow}

\begin{document}
\title{Positive enumerable functors}
\author{Barbara F.\ Csima\inst{1} \and Dino Rossegger\orcidID{0000-0003-3494-9049}\inst{2} \and Daniel
Yu\inst{3}}
\institute{Department of Pure Mathematics, University of Waterloo, Canada
  \email{csima@uwaterloo.ca} \and Department of Pure Mathematics, University of
  Waterloo, Canada \email{dino.rossegger@uwaterloo.ca} \and University of
Waterloo, Canada \email{zy3yu@uwaterloo.ca}}
\maketitle
\begin{abstract}
  We study reductions well suited to compare structures and
  classes of structures with respect to properties based on enumeration
  reducibility. We introduce the notion of a positive enumerable functor and
  study the relationship with established reductions based on functors and
  alternative definitions.
\end{abstract}
\section{Introduction}
In this article we study notions of reductions that let us compare classes of
structures with respect to their computability theoretic properties.
Computability theoretic reductions between classes of structures can be
formalized using effective versions of the category theoretic notion of
a functor. While computable functors have already been used in the 80's by
Goncharov~\cite{goncharov1980}, the formal investigation of this
notion was only started recently after R.\ Miller, Poonen, Schoutens, and
Shlapentokh~\cite{miller2018} explicitly used a computable functor to obtain
a reduction from the class of graphs to the class
of fields. Their result shows that fields are universal with respect to many
properties studied in computable structure theory.


In~\cite{rossegger2017} the third author studied effective versions of functors
based on enumeration reducibility and their relation to notions of
interpretability. There, it was shown that the existence of a computable
functor implies the existence of an enumerable functor effectively isomorphic
to it. In that article there also appeared an unfortunately incorrect claim that enumerable functors are
equivalent to a variation of effective interpretability, a notion equivalent to
computable functors~\cite{harrison-trainor2017}.
Indeed, it was later shown in
Rossegger's thesis~\cite{rosseggerthesis}, that the existence of a computable
functor implies the existence of an enumerable functor and thus enumerable
functors are equivalent to the original notion. Hence, enumerable functors are equivalent to the original
version of effective interpretability. In this paper we provide a simple proof of the latter result. 
It is not very surprising that enumerable and computable functors are
equivalent, as the enumeration operators
witnessing the effectiveness of an enumerable functor are given access to the
atomic diagrams of structures, which are total sets.

The main objective of this article is the study of positive enumerable functors,
an effectivization of functors that grants the involved enumeration operators access to
the positive diagrams of structures instead of their atomic diagrams. While
computable functors are well suited to compare structures with
respect to properties related to relative computability and the Turing
degrees, positive enumerable functors provide the right framework to compare structures with respect to
their enumerations and properties related to the enumeration
degrees.

The paper is organized as follows. In \cref{sec:defandfirstres}, we first show that computable functors and
enumerable functors are equivalent, and then begin the study of positive
enumerable functors and reductions based on them. We show that reductions by
positive enumerable bi-transformations preserve enumeration degree spectra,
a generalization of degree spectra considering all enumerations of a structure
introduced by Soskov~\cite{soskov2004}.
We then exhibit an example consisting of two structures which are computably
bi-transformable but whose enumeration degree spectra are different. This
implies that positive enumerable functors and computable functors are
independent notions. Towards the end of the section we compare different
possible definitions of positive enumerable functors
and extend our results to
reductions between arbitrary classes of structures.

\section{Computable and enumerable functors}\label{sec:defandfirstres}
In this article we assume that our structures are in a relational language
$(R_i)_{i\in\omega}$ where each $R_i$ has arity $a_i$ and the map $i\mapsto
a_i$ is computable. We furthermore only consider countable structures with universe $\omega$.
We view classes of structures as categories where the objects are
structures in a given language $\mc L$ and the morphisms are isomorphisms
between them. Recall that a functor $F:\mf C\ra \mf D$ maps structures from $\mf
C$ to structures in $\mf D$ and maps isomorphisms $f:\A\ra\B$ to
$F(f):F(\A)\ra F(\B)$ preserving composition and identity.

The smallest classes we consider are isomorphism classes of a single structure $\A$,
\[ Iso(\A)=\{ \B: \B\cong\A\}.\]
We will often talk about a functor from $\A$ to $\B$, $F:\A\ra \B$ when we mean
a functor $F:Iso(\A)\ra Iso(\B)$.
Depending on the properties that we want our functor to preserve we may use
different effectivizations, but they will all be of the following form.
Generally, an effectivization of a functor $F:\mf C\ra\mf D$ will consist of a pair of operators
$(\Phi,\Phi_*)$ and a suitable coding $C$ such that
\begin{enumerate}
  \item for all $\A\in \mf C$, $\Phi(C(\A))=C(F(\A))$,
  \item for all $\A,\B\in\mf C$ and $f\in Hom(\A,\B)$,
    $\Phi(C(\A),C(f),C(\B))= C(F(f))$.
\end{enumerate}
In this article the operators will either be enumeration or Turing operators.
If the coding is clear from context we will omit the coding function, i.e., we
write $\Phi(\A)$ instead of $\Phi(C(\A))$. The most
common coding in computable structure theory is the following.

\begin{definition}\label{def:atomicdiagram}
  Let $\A$ be a structure in relational language ${(R_i)}_{i\in \omega}$. Then the
  \emph{atomic diagram} $D(\A)$ of $\A$ is the set
  \[ \bigoplus_{i\in\omega} R_i^{\A}\oplus \bigoplus_{i\in\omega}\neg{R_i^{\A}}.\]
\end{definition}
In the literature one can often find different definitions of the atomic
diagram. It is easy to show that all of these notions are Turing and
enumeration equivalent. The reason why we chose this definition is that it is
conceptually easier to define the positive diagram and deal with enumerations
of structures like this. We are now ready to define various effectivizations of
functors.
\begin{definition}[\cite{miller2018},\cite{harrison-trainor2017}]\label{def:computablefunctor}
  A functor $F:\mf C\ra\mf D$ is \emph{computable} if there is a pair of Turing
  operators $(\Phi,\Phi_*)$ such that for all $\A,\B\in\mf C$
  \begin{enumerate}\tightlist%
    \item $\Phi^{D(\A)}=D(F(\A))$,
    \item for all $f\in Hom(\A,\B)$, $\Phi^{D(\A)\oplus Graph(f)\oplus
      D(\B)}=F(f)$.
  \end{enumerate}
\end{definition}
\begin{definition}\label{def:enumerablefunctor}
  A functor $F:\mf C\ra\mf D$ is \emph{enumerable} if there is a pair
  $(\Psi,\Psi_*)$ where $\Psi$ and $\Psi_*$ are enumeration
  operators such that for all $\A,\B\in\mf C$
  \begin{enumerate}\tightlist%
    \item $\Psi^{D(\A)}=D(F(\A))$,
    \item for all $f\in hom(\A,\B)$, $\Psi_*^{D(\A)\oplus Graph(f)\oplus
      D(\B)}=Graph(F(f))$.
  \end{enumerate}
\end{definition}
In~\cite{rossegger2017} enumerable functors were defined differently, using
a Turing operator instead of an enumeration operator for the homomorphisms.
The definition was as follows.
\begin{definition}[\cite{rossegger2017}]\label{def:starenumerablefunctor}
  A functor $F:\mf C\ra\mf D$ is \emph{$\star$-enumerable} if there is a pair
  $(\Psi,\Phi_*)$ where $\Psi$ is an enumeration operator and $\Phi_*$ is a Turing
  operator such that for all $\A,\B\in\mf C$
  \begin{enumerate}\tightlist%
    \item $\Psi^{D(\A)}=D(F(\A))$,
    \item for all $f\in hom(\A,\B)$, $\Phi_*^{D(\A)\oplus Graph(f)\oplus
      D(\B)}=Graph(F(f))$.
  \end{enumerate}
\end{definition}
It turns out that the two definitions are equivalent and we will thus stick
with \cref{def:enumerablefunctor} which seems to be more natural.
\begin{proposition}\label{prop:starenumiffenum}
  A functor $F:\A\ra\B$ is enumerable if and only if it is
  $\star$-enumerable.
\end{proposition}
\begin{proof}
Say we have an enumerable
functor given by $(\Psi,\Psi_*)$ and an isomorphism $f:\tilde\A\ra \hat \A$
for $\tilde \A\cong\hat\A\cong \A$.
We can compute the isomorphism $F(f)$ by enumerating $Graph(F(f))$ using
$\Psi_*^{\tilde\A\oplus
f\oplus \hat \A}$. For every $x$ we are guaranteed to enumerate
$(x,y)\in Graph(F(f))$ for some $y$ as the domain of $\A$ is $\omega$. This
is uniform in $\tilde\A$, $f$ and $\hat\A$. Thus there is a Turing operator $\Phi_*$
such that $(\Psi,\Phi_*)$ witnesses that $F$ is $\star$-enumerable.

Now, say $F$ is $\star$-enumerable as witnessed by $(\Psi,\Phi_*)$.
For every $\sigma,x,y$ with $\Phi^\sigma_*(x)\downarrow=y$ such that $\sigma$ can be split into $\sigma_0\oplus
\sigma_1\oplus\sigma_2$ where $\sigma_0$, $\sigma_2$ are partial characteristic
functions of finite structures in a finite sublanguage $L$ of the language of $\A$
and $\sigma_1$ is the partial graph of a function, consider the set
\begin{align*}X^{x,y}_\sigma=\{&(B\oplus Graph(\tau)\oplus C,\langle x,y\rangle): B, C
    \text{ are atomic diagrams of finite }\\
    & L\text{-structures}, B\text{ compatible with }\sigma_0,\,
    C\text{ compatible with }\sigma_2,\\
    & \sigma_1(u,v)=1\ra \tau(u)=v, \text{ and } \sigma_1(u,v)=0\ra \tau(u)=z\\
    & \text{where } z\not\in range(\sigma_1)
\}.\end{align*}
We can now define our enumeration operator as
$\Psi_\star=\bigcup_{x,y,\sigma:\Phi_\star^\sigma(x)\downarrow=y} X^{x,y}_\sigma$. Given an enumeration of
$\Phi_*$ we can produce an enumeration of $\Psi_*$, so $\Psi_*$ is
c.e. It remains to show that $\Psi_*^{\hat\A\oplus f\oplus \tilde
\A}=\Phi_*^{\hat \A\oplus f\oplus \tilde \A}$.

Say $\Phi_*^{\tilde\A\oplus f\oplus \hat \A}(x)=y$. Then there is
$\sigma\preceq \tilde\A\oplus f\oplus \hat \A$ such that $(\sigma,x,y)\in
\Phi_*$ and thus by the construction of $X_{\sigma}$ there is
$B\subseteq D(\tilde\A)$, $C\subseteq D(\hat\A)$ and $Graph(\tau)\subseteq
Graph(f)$ such that $(B\oplus Graph(\tau)\oplus C, \langle
x,y\rangle)\in X_\sigma$. Thus $\langle x,y\rangle\in\Psi_*^{\tilde\A\oplus
f\oplus \hat \A}$.

On the other hand say $\langle x,y\rangle\in\Psi_*^{\tilde\A\oplus
f\oplus \hat \A}$. Then, there is $(B\oplus Graph(\tau)\oplus
C,\langle x,y\rangle)\in \Psi_*$ with $B\subseteq \tilde
\A$, $Graph(\tau)\subseteq Graph(f)$ and $C\subseteq \hat\A$.
Furthermore, there is $\sigma\preceq \chi_{B\oplus Graph(\tau)\oplus
C}$ such that $(\sigma,x,y)\in \Phi_*$. Thus $\Psi_*^{\tilde\A\oplus f\oplus
\hat \A}=Graph(F(f))$ for any $\hat \A\cong \tilde \A\cong \A$ and
$f:\tilde\A\cong\hat\A$ and hence $F$ is enumerable.
\qed\end{proof}

In~\cite{rossegger2017} it was shown that the existence of an enumerable
functor implies the existence of a computable functor and
in~\cite{rosseggerthesis} the converse was shown. We give a simple
proof of the latter.
\begin{theorem}\label{thm:compimpenum}
  If $F:\A\ra\B$ is a computable functor, then it is enumerable.
\end{theorem}
\begin{proof}
 Given a computable functor $F$ we will show that $F$ is $\star$-enumerable.
 That $F$ is then also enumerable follows from \cref{prop:starenumiffenum}.

 Let $D(L_\mathcal{A})$ be the collection of finite atomic diagrams in the language of $\mathcal{A}$.
 To every $p\in D(L_\mathcal{A})$ we associate a finite string
 $\alpha_p$ in the alphabet $\{0,1,\uparrow\}$ so that if $p$ specifies that
 $R_i$ holds on elements coded by $u$, then we set that $\neg R_i$ does not
 hold on these elements. More formally, $\alpha_p(x)=1$ if $x\in
 p$, $\alpha_p(x)=0$  if $x=2\langle i, u\rangle$ and $2\langle i,
 u\rangle+1\in p$ or $x=2\langle i,u\rangle +1$ and $2\langle i, u\rangle \in
 p$
 , and $\alpha_p(x)=\uparrow$ if $x$ is less than
 the largest element of $p$ and none of the other cases fits. We also
 associate a string $\tilde\alpha_p\in 2^{|\alpha_p|}$ with $p$ where $\tilde \alpha_p(x)=1$
 if and only if $\alpha_p(x)=1$ and $\tilde \alpha_p(x)=0$ if and only if $\alpha_p(x)=0$ or
 $\alpha_p(x)\uparrow$.

 Let the computability of $F$ be witnessed by $(\Phi,\Phi_*)$. We build the
 enumeration operator $\Psi$ as follows. For every $p\in D(L_\A)$ and every
 $x$ if $\Phi^{\tilde\alpha_p}(x)\downarrow=1$ and every call to the oracle
 during the computation is on an element $z$ such that $\alpha_p(z)\neq
 \uparrow$, then enumerate $(p,x)$ into $\Psi$. This finishes the construction
 of $\Psi$.

 Now, let $\hat \A\cong \A$. We have that $x\in \Psi^{\hat \A}(x)$ if and only if there exists
 $p\in D(L_\A)$ such that $p \subseteq D(\hat \A)$ and $(p,x)\in \Psi$.
 We further have that $(p,x)\in\Psi$ if and only if
 $\Phi^{\tilde\alpha_p}(x)\downarrow=1$ if and only if $\Phi^{\hat\A}(x)=1$.
 Thus $F$ is enumerable using $(\Psi,\Phi_*)$.
\qed\end{proof}

Combining \cref{thm:compimpenum} with the results from~\cite{rossegger2017} we
obtain that enumerable functors and computable functors
defined using the atomic diagram of a structure as input are
equivalent notions. This is not surprising. After all, the atomic diagram of
a structure
always has total enumeration degree and there is a canonical isomorphism between
the total enumeration degrees and the Turing degrees.
In order to make this equivalence precise we need another definition.

\begin{definition}[\cite{harrison-trainor2017}]
A functor $F: \mathfrak{C} \to \mathfrak{D}$ is  \emph{effectively isomorphic} to a functor
    $G: \mathfrak{C} \to \mathfrak{D}$ if there is a Turing functional $\Lambda$ such that for any
    $\mathcal{A} \in  \mathfrak{C}$, $\Lambda^\mathcal{A} : F(\mathcal{A}) \to G(\mathcal{A})$ is an isomorphism.
Moreover, for any morphism $h\in Hom(\mathcal{A}, \mathcal{B})$ in $\mathfrak{C}$,
    $\Lambda^\mathcal{B} \circ F(h) = G(h) \circ \Lambda^\mathcal{A}$.
That is, the diagram below commutes.
\begin{center}
  \begin{tikzpicture}
    \node (FA) at (0,2) {$F(\mathcal{A})$};
    \node (FB) at (0,0) {$F(\mathcal{B})$};
    \node (GA) at (3,2) {$G(\mathcal{A})$};
    \node (GB) at (3,0) {$G(\mathcal{B})$};

    \draw[->] (FA) -- node[above] {$\Lambda^\mathcal{A}$} (GA);
    \draw[->] (FB) -- node[above] {$\Lambda^\mathcal{B}$} (GB);
    \draw[->] (FA) -- node[left] {$F(h)$} (FB);
    \draw[->] (GA) -- node[right] {$G(h)$} (GB);
\end{tikzpicture}
\end{center}
\end{definition}
The following is an immediate corollary of \cref{thm:compimpenum} and~\cite[Theorem 2]{rossegger2017}.
  \begin{theorem}\label{thm:enumiffcomp}
  Let $F:\A\ra \B$ be a functor. Then $F$ is computable if and only if there is
  an enumerable functor $G:\A\ra\B$ effectively isomorphic to $F$.
\end{theorem}

\begin{definition}[\cite{harrison-trainor2017}]
Suppose $F: \mathfrak{C} \to \mathfrak{D}$, $G: \mathfrak{D} \to \mathfrak{C}$ are functors such that
    $G \circ F$ is effectively isomorphic to $Id_\mathfrak{C}$ via the Turing functional
    $\Lambda_\mathfrak{C}$ and $F \circ G$ is effectively isomorphic to $Id_\mathfrak{D}$
    via the Turing functional $\Lambda_\mathfrak{D}$.
If furthermore, for any $\mathcal{A} \in \mathfrak{C}$ and $\mathcal{B} \in \mathfrak{D}$,
    $\Lambda_\mathfrak{D}^{F(\mathcal{A})}= F(\Lambda_\mathfrak{C}^{\mathcal{A}}) :
    F(\mathcal{A}) \to F(G(F(\mathcal{A})))$ and
    $\Lambda_\mathfrak{C}^{G(\mathcal{B})}=
    G(\Lambda_\mathfrak{D}^{\mathcal{B}}) : G(\mathcal{B}) \to G(F(G(\mathcal{B})))$,
    then $F$ and $G$ are said to be \emph{pseudo inverses}.
\end{definition}
\begin{definition}[\cite{harrison-trainor2017}]\label{def:compbitrans}
  Two structures $\A$ and $\B$ are \emph{computably bi-transformable} if there
  are computable pseudo-inverse functors $F:\A\ra\B$ and $G:\B\ra\A$.
\end{definition}
If the functors in \cref{def:compbitrans} are enumerable instead
of computable then we say that $\A$ and $\B$ are \emph{enumerably
bi-transformable}. As an immediate corollary of \cref{thm:enumiffcomp} we
obtain the following.
\begin{corollary}\label{cor:bitransenumiffcomp}
  Two structures $\A$ and $\B$ are enumerably bi-transformable if and only if
  they are computably bi-transformable.
\end{corollary}
\section{Effectivizations using positive diagrams}\label{sec:positivefunctors}
We now turn our attention to the setting where we only have positive
information about the structures.
We follow Soskov~\cite{soskov2004} in our definitions. See
also the survey paper by Soskova and Soskova~\cite{soskova2017} on computable structure theory and enumeration degrees.
\begin{definition}\label{def:positivediagram}
  Let $\A$ be a structure in relational language ${(R_i)}_{i\in\omega}$. The
  \emph{positive diagram of $\A$}, denoted by $P(\A)$, is the set
  \[ =\oplus \neq\oplus\bigoplus_{i\in\omega} R_i^{\A}.\]
\end{definition}
We are interested in the degrees of enumerations of $P(\A)$. To be more precise
let $f$ be an enumeration of $\omega$ and for $X\subseteq \omega^n$ let
\[ f^{-1}(X)=\{ \langle x_1,\dots,x_n\rangle: ( f(x_1),\dots,
f(x_n))\in X\}.\]
Given $\A$ let $f^{-1}(\A)=f^{-1}(=)\oplus f^{-1}(\neq)\oplus f^{-1}(R_0^{\A})\oplus
\dots$. Notice that if $f=id$, then $f^{-1}$ is just the positive diagram of $\A$.
\begin{definition}\label{def:enumerationspectrum}
  The \emph{enumeration degree spectrum} of $\A$ is the set
    \[ eSp(\A)=\{ d_e(f^{-1}(\A)): f \text{ is an enumeration of } \omega\}.\]
  If $\mathbf a$ is the least element of $eSp(\A)$, then $\mathbf a$ is called the
  \emph{enumeration degree} of $\A$.
\end{definition}
In order to obtain a notion of reduction that preserves enumeration spectra we
need an effectivization of functors where we use positive
diagrams of structures as coding. It is clear that for computable
functors this makes no difference as $P(\A)\equiv_T D(\A)$. For enumerable
functors it does make a difference. We also need to replace the Turing operators in the definition of pseudo
inverses with enumeration operators. The new notions are as follows.
\begin{definition}\label{def:positiveenumerablefunctor}
  A functor $F:\mf C\ra\mf D$ is \emph{positive enumerable} if there is a pair
  $(\Psi,\Psi_*)$ where $\Psi$ and $\Psi_*$ are enumeration
  operators such that for all $\A,\B\in\mf C$
  \begin{enumerate}\tightlist%
    \item $\Psi^{P(\A)}=P(F(\A))$,
    \item for all $f\in hom(\A,\B)$, $\Psi_*^{P(\A)\oplus Graph(f)\oplus
      P(\B)}=Graph(F(f))$.
  \end{enumerate}
\end{definition}
\begin{definition}
A functor $F: \mathfrak{C} \to \mathfrak{D}$ is  \emph{enumeration isomorphic} to a functor
    $G: \mathfrak{C} \to \mathfrak{D}$ if there is an enumeration operator $\Lambda$ such that for any
    $\mathcal{A} \in  \mathfrak{C}$, $\Lambda^{P(\mathcal{A})} : F(\mathcal{A}) \to G(\mathcal{A})$ is an isomorphism.
Moreover, for any morphism $h\in Hom(\mathcal{A}, \mathcal{B})$ in $\mathfrak{C}$,
$\Lambda^{P(\mathcal{B})} \circ F(h) = G(h) \circ \Lambda^{P(\mathcal{A})}$.
%
\end{definition}
\begin{definition}
Suppose $F: \mathfrak{C} \to \mathfrak{D}$, $G: \mathfrak{D} \to \mathfrak{C}$ are functors such that
    $G \circ F$ is enumeration isomorphic to $Id_\mathfrak{C}$ via the
    enumeration operator
    $\Lambda_\mathfrak{C}$ and $F \circ G$ is enumeration isomorphic to $Id_\mathfrak{D}$
    via the enumeration operator $\Lambda_\mathfrak{D}$.
If, furthermore, for any $\mathcal{A} \in \mathfrak{C}$ and $\mathcal{B} \in \mathfrak{D}$,
$\Lambda_\mathfrak{D}^{P(F(\mathcal{A}))}=
F(\Lambda_\mathfrak{C}^{P(\mathcal{A})}) :
    F(\mathcal{A}) \to F(G(F(\mathcal{A})))$ and
    $\Lambda_\mathfrak{C}^{P(G(\mathcal{B}))}=
    G(\Lambda_\mathfrak{D}^{P(\mathcal{B})}) : G(\mathcal{B}) \to G(F(G(\mathcal{B})))$,
    then $F$ and $G$ are said to be \emph{enumeration pseudo inverses}.
\end{definition}
\begin{definition}\label{def:posenumbitrans}
  Two structures $\A$ and $\B$ are \emph{positive enumerably bi-transform-able} if there
  are positive enumerable enumeration pseudo-inverse functors $F:\A\ra\B$ and $G:\B\ra\A$.
\end{definition}
\begin{theorem}
  Let $\A$ and $\B$ be positive enumerably bi-transformable. Then
  $eSp(\A)=eSp(\B)$.
\end{theorem}
\begin{proof}
  Say $\A$ and $\B$ are positive enumerably bi-transformable by $F:\A\ra\B$
  and $G:\B\ra\A$. Let $f$ be an arbitrary enumeration of $\omega$, then,
  viewing $f^{-1}(\A)/f^{-1}(=)$ as a structure on $\omega$ by pulling back
    a canonical
    enumeration of the least elements in its $=$-equivalence classes, we have
    that there is $\hat \A\cong \A$ such that $P(\hat \A)=f^{-1}(\A)/f^{-1}(=)$ and $P(\hat\A)\leq_e
  f^{-1}(\A)$. As $F$ is positive enumerable we have that $f^{-1}(\A)\geq_e
  P(F(\hat \A))$.
  Furthermore, we shall see that $f^{-1}(F(\hat\A))\leq_e f^{-1}(\A)$ and that
  $f^{-1}(\A)/f^{-1}(=)=P(F(\hat\A))$.
  Given an enumeration of $f^{-1}(\A)$ and an enumeration of
  $P(F(\hat\A))$, we may first order the equivalence classes of $f^{-1}(=)$ by
  their least elements and then, if $R_i(a_1,\dots,a_n)\in P(F(\hat\A))$ we
  enumerate $R_i(b_1,\dots,b_n)$ for all $b_1,\dots,b_n\in\omega$ such that $b_j$ is in
  the $a_j$\textsuperscript{th} equivalence class of $f^{-1}(=)$. It is not
  hard to see that this gives an enumeration of a set $X$ such that $f^{-1}(=)\oplus f^{-1}(\neq)\oplus
  X=f^{-1}(F(\hat\A))$, that $f^{-1}(F(\hat\A))/f^{-1}(=)=P(F(\hat\A))$, and since by construction $f^{-1}(F(\hat\A))\leq_e P(F(\hat\A))\oplus f^{-1}(\A)$ we have $f^{-1}(F(\hat\A))\leq_e f^{-1}(\A)$.

  We can apply the same argument with $G$ in place of $F$ and $F(\hat\A)$ in
  place of $\A$ to get that $f^{-1}(G(F(\hat\A)))/f^{-1}(=)=P(G(F(\hat\A)))$ and
  \[ f^{-1}(G(F(\hat \A)))\leq_e f^{-1}(F(\hat\A))\leq_e f^{-1}(\A).\]
  At last, recall that, as $\A$ and $\B$ are positive enumerably
  bi-transformable, there is an enumeration operator $\Psi$ such that
    $\Psi^{P(G(F(\hat \A)))}$ is the enumeration of the graph of an isomorphism
  $i:G(F(\hat\A))\cong\hat \A$. But then $(f\circ i)^{-1}(G(F(\hat
  \A)))=f^{-1}(\A)$ and
  \[ f^{-1}(\A)\leq_e f^{-1}(G(F(\hat \A)))\leq_e f^{-1}(F(\hat\A))\leq_e f^{-1}(\A).\]
  This shows that $eSp(\A)\subseteq eSp(\B)$. The proof that $eSp(\B)\subseteq
  eSp(\A)$ is analogous.
\qed\end{proof}
\begin{proposition}\label{prop:compfunctors_noeSp}
  There are computably bi-transformable structures $\A$ and $\B$ such that
  $eSp(\A)\neq eSp(\B)$. In particular, $\A$ and $\B$ are not positive
  enumerably bi-transformable.
\end{proposition}
\begin{proof}
  Let $\A=(\omega,\underline 0,s, K)$ where $s$ is the successor relation
  on $\omega$, $\underline 0$ the first element, and $K$ the membership
  relation of the halting set. Assume
  $\B=(\omega,\underline 0,s,\bar K)$ is defined as $\A$
  except that $\bar K(x)$ if and only if $\neg K(x)$. There is a computable
  functor $F:\A \to \B$ taking $\hat \A=(\omega, \underline 0^{\hat{\A}},
  s^{\hat{\A}}, K^{\hat{\A}})\cong \A$ to $F(\hat\A)=(\omega,
  \underline 0^{\hat{\A}},
  s^{\hat{\A}}, \neg K^{\hat{\A}})$ and acting as the identity on
  isomorphisms. Furthermore, $F$ has a computable inverse and thus $\A$
is computably bi-transformable to $\B$.

  However, $\A$ has enumeration degree $\mbf{0}'_e$ and $\B$
has enumeration degree $\bar{\mbf{0}'}_e$. Thus there cannot be a positive
  enumerable functor from $\B$ to $\A$.
\qed\end{proof}
The following shows that computable functors and positive enumerable functors
are independent.
\begin{proposition}\label{prop:posenumnotcomp}
  There are structures $\A$ and $\B$ such that $\A$ is positive enumerably
  bi-transformable with $\B$ but $\A$ is not computably bi-transformable with
  $\B$.
\end{proposition}
\begin{proof}
  Let $\A$ be as in \cref{prop:compfunctors_noeSp}, i.e.,
  $\A=(\omega,\underline 0,s, K)$ and $\B=(\omega, \underline 0,s)$. Then it is
  not hard to see that $\A$ is positive enumerably bi-transformable with
  $\B$. However, there can not be a computable functor from $\B$ to $\A$ as
  $\B$ has Turing degree $\mbf 0$ and $\A$ has Turing degree $\mbf 0'$.
\qed\end{proof}

We have seen in \cref{prop:starenumiffenum} that $\star$-enumerable functors
and enumerable functors are equivalent. Positive enumerable functors also admit
a different definition.
\begin{definition}\label{def:posstarenum}
  A functor $F:\mf C\ra\mf D$ is \emph{positive $\star$-enumerable} if there is a pair
  $(\Psi,\Phi_*)$ where $\Psi$ is an enumeration operator and $\Phi_*$ is a Turing
  operator such that for all $\A,\B\in\mf C$
  \begin{enumerate}\tightlist%
    \item $\Psi^{P(\A)}=P(F(\A))$,
    \item for all $f\in hom(\A,\B)$, $\Phi_*^{P(\A)\oplus Graph(f)\oplus
      P(\B)}=Graph(F(f))$.
  \end{enumerate}
\end{definition}
\begin{proposition}
  Every positive enumerable functor is positive $\star$-enumerable.
\end{proposition}
\begin{proof}
  Let $F:\A\ra \B$ be given by $(\Psi,\Psi_*)$ and let $f:\tilde\A\cong
  \hat\A$ for $\tilde\A\cong\hat\A\cong \A$. Now we can define a procedure
  computing $F(f)$ as follows. Given $x$, and $\tilde\A\oplus f\oplus\hat \A$
  enumerate $\Psi_*^{\tilde\A \oplus f\oplus \hat \A}$ until $\langle
  x,y\rangle\searrow
  \Psi_*^{\tilde\A \oplus f\oplus \hat \A}$ for some $y$. This is uniform in
  $\tilde\A\oplus f\oplus\hat\A$ and thus there exists a Turing operator $\Phi_*$ with
  this behaviour. The pair $(\Psi,\Phi_*)$ then witnesses that $F$ is
  $\star$-enumerable.
\qed\end{proof}
\begin{theorem}
  There is positive $\star$-enumerable functor that is not enumeration
  isomorphic to any positive enumerable functor.
\end{theorem}
\begin{proof}
  We will build two structures $\A$ and $\B$ such that there is a positive
  $\star$-enumerable functor $F:\A\ra\B$ that is not positive enumerable.
  The structure $\A$ is a graph constructed as follows. It has a vertex $a$
  with a loop connected to $a$ and a cycle of size $n$ for every natural number
  $n$. If $n\in\emptyset'$ then there is an edge between $a$ and one element of
  the $n$ cycle, otherwise there is no such edge. Clearly,
  $deg_T(P(\A))=\mathbf 0'$ and $P(\A)\not\geq_e \bar\emptyset'$.

  The structure $\B$ is a typical graph that
  witnesses that there is a structure with degree of categoricity $\mathbf 0'$
  (that is, $\mathbf 0'$ is the least degree computing an isomorphism between
  any two computable copies of $\B$).
  Let us recap how we build two copies of $\mc B$, $\mc B_1$ and $\mc B_2$
  such that $\mathbf 0'$ is the least degree computing isomorphism between $\mc
  B_1$ and $\mc B_2$. Both graphs consist of an infinite ray with a loop at
  its first element. Let $v_i$ be the $i^{th}$ element in the ray in $\mc B_1$
  and $\hat v_i$ be the $i^{th}$ element in the ray in $\mc B_2$. Now for every
  $v_i$ there are two elements $a_i$ and $b_i$ with $v_iEa_i$ and $v_iE
  b_i$. Likewise for every $\hat v_i$ there are two elements $\hat a_i$ and
  $\hat b_i$ with $\hat v_iE\hat a_i$ and $\hat v_iE \hat b_i$. Furthermore
  there are additional vertices $s_i$, $\hat s_i$ with $a_iE s_i$ and $\hat
  a_iE \hat s_i$.

  Take an enumeration of $\emptyset'$. If $i\searrow \emptyset'$, then add
  vertices $b_iE
  \cdot E\cdot$ and $\hat s_i E \cdot$, $\hat b_i E \cdot$. This finishes the construction of $\B$. It is not hard
  to see that there is a unique isomorphism $f:\B_1\ra \B_2$ and that
  $deg(f)=\mathbf 0'$ and $Graph(f)\geq_e \bar\emptyset'$. 

  We now construct the functor $F$. Given an enumeration of $P(\hat \A)$
  for $\hat \A\cong \A$ we wait until we see the cycle containing $0$ (any
  natural number would work). If it
  is of even length, or $0$ is the special vertex $a$, we let $F(\hat\A)=\B_1$ and if it is
  of odd length we let $F(\hat\A)=\B_2$. Clearly given any
  enumeration of a copy of $\A$ this procedure produces an enumeration of
  a copy of $\B$.

  As $\B$ is rigid we just let $F(f:\hat\A\ra\tilde
  \A)=g:F(\hat\A)\ra F(\tilde \A)$ where $g$ is the unique isomorphism between
  $F(\hat \A)$ and $F(\tilde \A)$. Note that there is a Turing operator
  $\Theta$ such that $\Theta^{P(\hat\A)}=\emptyset'$ for any $\hat\A\cong \A$
  and that the isomorphism between $F(\hat\A)$ and $F(\tilde \A)$ can be
  computed uniformly from $P(F(\hat\A))\oplus P(F(\tilde\A))\oplus \emptyset'$.
  Thus, there is an operator
  $\Phi_*$ witnessing that $F$ is positive $\star$-enumerable.

  To see that $F$ is not positive enumerable consider two copies $\hat \A$ and
  $\tilde \A$ of $\A$ with $deg_e(P(\hat \A))=deg_e(P(\tilde\A))=\mathbf 0'_e$ such
  that $0$ is part of an even cycle in $\hat \A$ and part of an odd cycle in
  $\tilde \A$.
  Notice that there is $f:\hat\A\ra\tilde\A$ such that $P(\hat\A)\oplus
  P(\tilde\A)\geq_e P(\hat\A)\oplus
  Graph(f:\hat\A\ra\tilde\A)\oplus P(\tilde \A)$, and also that
  $P(\hat\A)\oplus P(\tilde
  \A)\not\geq_e \bar \emptyset'$. But $Graph(g:F(\hat\A)\ra F(\tilde\A))\geq_e
  \bar \emptyset'$ as $F(\hat \A)=\B_1$ and $F(\tilde\A)=\B_2$. Thus there can not be an enumeration operator witnessing
  that $F$ is positive enumerable.
  
  Assume $F$ was enumeration isomorphic to a positive
  enumerable functor $G$ and that this isomorphism is witnessed by
  $\Lambda$. Then, taking $\hat \A$, $\tilde\A$ and $f:\hat \A\ra \tilde \A$ as
  in the above paragraph we have that
  $P(\hat\A)\oplus P(\tilde\A)\geq_e Graph(G(f))$. But then
  \[P(\hat\A)\oplus P(\tilde\A)\geq_e Graph(\Lambda^{P(\hat\A)}\circ G(f)\circ {\Lambda^{P(\tilde
  \A)}}^{-1})=Graph(F(f))\geq_e \bar\emptyset'.\]
This is a contradiction since $deg_e(P(\hat\A)\oplus P(\tilde\A))=\mathbf 0'_e$.
\qed\end{proof}
\section{Reductions between arbitrary classes}
So far we have seen how we can compare structures with respect to computability
theoretic properties. Our notions can be naturally extended to allow the
comparison of arbitrary classes of structures.
\begin{definition}[\cite{harrison-trainor2017}]
  Let $\mf C$ and $\mf D$ be classes of structures. The class $\mf C$ is
  \emph{uniformly computably transformably reducible}, short u.c.t.\
  reducible, to $\mf D$ if there are a subclass $\mf D'\subseteq\mf D$ and computable functors $F:\mf C\ra \mf
  D'\subseteq \mf D$ and $G:\mf D'\ra \mf C$ such that $F$ and $G$ are
  pseudo-inverses.
\end{definition}
\begin{definition}
  Let $\mf C$ and $\mf D$ be classes of structures. The class $\mf C$ is
  \emph{uniformly (positive) enumerably transformably reducible}, short u.e.t.,
  (u.p.e.t.)
  reducible, to $\mf D$ if there is a subclass $\mf D'\subseteq\mf D$ and
  (positive) enumerable functors $F:\mf C\ra \mf
  D'\subseteq \mf D$ and $G:\mf D'\ra \mf C$ such that $F$ and $G$ are
  pseudo-inverses.
\end{definition}
\Cref{prop:compfunctors_noeSp,prop:posenumnotcomp} show that u.p.e.t.\ and
u.c.t\ reductions are independent notions.
\begin{corollary}
  There are classes of structures $\mf C_1,\mf C_2$ and $\mf D_1,\mf D_2$ such
  that
  \begin{enumerate}
    \tightlist%
  \item $\mf C_1$ is u.c.t.\ reducible to $\mf D_1$ but $\mf C_1$ is not
    u.p.e.t.\ reducible to $\mf D_1$.
  \item $\mf C_2$ is u.p.e.t.\ reducible to $\mf D_2$ but $\mf C_2$ is not
    u.c.t.\ reducible to $\mf D_2$.
  \end{enumerate}
\end{corollary}
Similar to \cref{cor:bitransenumiffcomp} we obtain the equivalence of u.e.t.\ and u.c.t\ reductions.
\begin{corollary}
  Let $\mf C$ and $\mf D$ be arbitrary classes of countable structures. Then
  $\mf C$ is u.e.t.\ reducible to $\mf D$ if and only if it is u.c.t.\ reducible
  to $\mf D$.
\end{corollary}
\printbibliography
\end{document}